\documentclass[11pt]{article}%
\usepackage{graphics,mathrsfs}
\usepackage{amsmath}
\usepackage{booktabs}
\usepackage{mathtools}
\usepackage{enumitem}
\usepackage{epsfig}
\usepackage{graphicx}
\usepackage{amssymb}
\usepackage{amsmath,bm}
\usepackage{caption}
\usepackage{subcaption}
\usepackage{subfloat}
\usepackage{bbold}
\usepackage{comment}

\usepackage{geometry}
\newtheorem{theorem}{\bf Theorem}[section]
\newtheorem{proposition}{\bf Proposition}[section]
\newtheorem{lemma}{\bf Lemma}[section]
\newtheorem{corollary}{\bf Corollary}[section]
\newtheorem{remark}{\bf Remark}[section]
\newtheorem{definition}{\bf Definition}[section]

\newtheorem{example}{\bf Example }[section]

\usepackage{color}

\newenvironment{proof}{
	\begin{trivlist}
		\item[\hspace{\labelsep}{\bf\noindent Proof. }] }{\par\hfill $\Box$\end{trivlist}
	\par}

\date{\plain}

\title{
	\huge\bf Generalized Gini's mean difference through distortions and copulas, and related minimizing problems 
	\date{\today}
}

\author{
	\large   
	Marco {\bf Capaldo}$^1$\thanks{ORCID: 0000-0002-0255-8935} \qquad
	Antonio {\bf Di Crescenzo}$^1$\thanks{ORCID: 0000-0003-4751-7341} \qquad 
	Franco {\bf Pellerey} $^2$ \thanks{ORCID: 0000-0002-8983-855X}\\
	\small $^1$ Dipartimento di Matematica, Universit\`a degli Studi di Salerno \\
	\small Via Giovanni Paolo II, 132, I-84084 Fisciano (SA), Italy \\
	\small $^2$ Dipartimento di Scienze Matematiche, Politecnico di Torino \\
	\small Corso Duca degli Abruzzi, 24, I-10129 Torino, Italy \\
	\normalsize Email:  \{mcapaldo, adicrescenzo\}@unisa.it,
	franco.pellerey@polito.it  
}

\begin{document}
	\maketitle
	\abstract{Given a random variable $X$ and considered a family of its possible distortions, we define two new measures of distance between $X$ and each its distortion. For these distance measures, which are extensions of the Gini's mean difference, conditions are determined for the existence of a minimum, or a maximum, within specific families of distortions, generalizing some results presented in the recent literature. 
		
	\medskip\noindent
	{\bf Mathematics Subject Classification}: 60E05, 60E15, 62C05  
	}
		
	\medskip
	\noindent{\bf Keywords:} Distance metrics;  Distortion function; Copula; Proportional hazard model; Proportional reversed hazard model; Gini's mean difference.
	 
\section{Introduction and background}
In the recent literature, minimizing problems of distance between random variables based on suitable measures have been examined (see, for instance, Ortega-Jim\'enez et al.\ \cite{patricia:2023}, Delbaen and Majumdar \cite{DelbaenMajumdar}, and references therein). 
This paper fits into this line of research, presenting new general measures of distance from a random variable $X$ and its distortions, which can be possibly dependent on $X$, and describing conditions for such measures to present a minimum, or a maximum, within specific classes of distortions. For the definition of such distance measures (given in Eqs.\ (\ref{eq:distorted GMD}) and (\ref{eq:copula-distorted GMD}) below), we provide in advance some assumptions on the variables considered, and some useful definitions. 
\par
Throughout the paper, we will assume that all random variables are absolutely continuous and that
all distribution functions and copulas are continuously differentiable. In addition, the terms increasing and decreasing will be used in strict sense. 
Let $X$ be a random variable with cumulative distribution function (c.d.f.) $F(x)=\mathsf{P}(X\le x)$ and survival function (s.f.) $\overline{F}(x)= 1-F(x)$, for $x\in\mathbb{R}$. Let 
$$
l=\inf\{x\in\mathbb{R} : F(x)>0\}, \qquad r=\sup\{x\in\mathbb{R} : \overline{F}(x)>0\}
$$
denote respectively the lower and upper limits of the support of $X$, which may be finite or infinite.
Moreover, given a c.d.f.\ $F(x)$, we will denote with  
$F^{-1}(u)=\sup\{x\colon F(x)\leq u\}$, $u\in [0,1]$, the right-continuous version of its inverse, also named quantile function in statistical framework.  
We recall that, an increasing function $h:[0,1]\rightarrow[0,1]$, such that $h(0)=0$ and $h(1)=1$ is called distortion function. 
The quantities $h(F(x))$ and $h(\overline{F}(x))$ are called distorted c.d.f. and distorted s.f., respectively. 
For more details about distortion functions see for instance Section 2.4 in Navarro \cite{navarro:2022}.
Let $(X,Y)$ be a random vector with joint c.d.f.\ $F_{(X,Y)}(x,y)=\mathsf{P}(X\le x, Y\le y)$, joint s.f.\ $\overline{F}_{(X,Y)}(x,y)=\mathsf{P}(X> x, Y> y)$, marginal c.d.f.'s $F_X(x),F_Y(y)$ and marginal s.f.'s $\overline{F}_X(x)$, $\overline{F}_Y(y)$, for $x,y\in\mathbb{R}$. Thanks to Sklar theorem, one has
$$
F_{(X,Y)}(x,y)=C(F_X(x),F_Y(y)),\qquad\overline{F}_{(X,Y)}(x,y)=\widehat{C}(\overline{F}_X(x),\overline{F}_Y(y)),\qquad x,y\in\mathbb{R},
$$ 
where $C$ and $\widehat{C}$ are copula distribution function and survival copula of $(X,Y)$, respectively. For more details about copula functions see Nelsen \cite{nelsen:2007}. 
We will consider only distorted s.f.'s and survival copulas.
In addition, we recall that, if $X$ and $Y$ are nonnegative random variables with c.d.f.'s $F$ and $G$, respectively, then  $X$ is said to be smaller than $Y$ in the dispersive order, denoted as $X\le_d Y$, if and only if 
(see Section 3.B of Shaked and Shanthikumar \cite{shaked:2007}) 
\begin{equation*}
	F^{-1}(v)-F^{-1}(u) \ge G^{-1}(v)- G^{-1}(u)\qquad
	\hbox{whenever }0<u \leq v<1. 
\end{equation*}
Moreover, if $X$ and $Y$ are absolutely continuous with density functions (d.f.'s) $f$ and $g$, respectively, then  
\begin{equation*}
	X\le_d Y     \quad \hbox{if and only if} \quad f(F^{-1}(u))\ge g(G^{-1}(u))\quad\forall \, u\in(0,1).
\end{equation*}
We recall that, the Gini's mean difference of a random variable $X$ is defined as
\begin{equation}\label{eq:GMD}
\mathsf{GMD}(X)=\mathsf{E}|X-X'|=2\int_{l}^{r}F(x)\overline{F}(x){\rm d}x,
\end{equation}
where $X'$ is an independent copy of $X$.
It represents a special case of some information measures, such as the generalized cumulative entropies and extropies considered in Kattumannil et al.\ \cite{Kattumannil:etal}.
In addition we mention that the following quantities 
\begin{equation}\label{eq:rmd,gini index}
	\mathsf{RMD}(X)=\frac{\mathsf{GMD}(X)}{\mathsf{E}(X)},\qquad\mathsf{G}(X)=\frac{\mathsf{RMD}(X)}{2},
\end{equation}
denote respectively the relative Gini's mean difference and the Gini's index of $X$, provided that the mean is finite and non-zero. 
\par 
Let us now consider a family of distortion functions
\begin{equation}\label{eq:family distortions}
	\mathcal{H}_\alpha=\{h_\alpha:[0,1]\rightarrow[0,1],\,\alpha\in\mathcal{I}\subset\mathbb{R}\}.	
\end{equation}
We will denote with $X_\alpha$ the random variable with s.f.\ $\overline{F}_{X_\alpha}(x)=h_\alpha\left(\overline{F}(x)\right)$, for $x\in\mathbb{R}$, and with  
\begin{equation}\label{eq:family}
	\mathcal{F}_{X,h_\alpha}=\{X_\alpha :\overline{F}_{X_\alpha}(x)=h_\alpha\left(\overline{F}(x)\right),\; \alpha\in\mathcal{I},\; x\in\mathbb{R}\}
\end{equation}
the family of all random variables with distorted s.f.'s from $X$ through $h_\alpha\in\mathcal{H}_\alpha$.
\par
We extend the Gini's mean difference in Eq.\ (\ref{eq:GMD}) in the cases in which the random variables are not necessarily independent or identically distributed. 
In particular, 
for $X$ and $X_\alpha$ independent random variables, we define the \emph{distorted Gini's mean difference of $X$} as
\begin{equation}\label{eq:distorted GMD}
	\eta_X(\alpha)\coloneqq\mathsf{E}|X-X_\alpha|=\int_{l}^{r}\left\{\overline{F}(x)+h_\alpha(\overline{F}(x))\left[1-2\overline{F}(x)\right]\right\}{\rm d}x.
\end{equation}
For any $\alpha\in\mathcal{I}$, given the distortion function $h_\alpha\in\mathcal{H}_\alpha$, let $(X, X_\alpha)$ be a random vector with survival copula in the parametric family
\begin{equation*}
	\mathcal{\hat{C}}^\theta=\{\widehat{C}_\theta\left(u,v\right);\,\theta\in\Theta,\,u,v\in[0,1]\}.
\end{equation*}
By taking into account that $|x_1-x_2|=x_1+x_2-2\min\{x_1,x_2\}$, we define the \emph{copula-distorted Gini's mean difference of $X$} as 
\begin{equation}\label{eq:copula-distorted GMD}
	\nu_X(\theta,\alpha)\coloneqq\mathsf{E}_{\widehat{C}_\theta}|X-X_\alpha|=\int_{l}^{r}\left\{\overline{F}(x)+h_\alpha(\overline{F}(x))-2\,\widehat{C}_\theta\left(\overline{F}(x),h_\alpha\left(\overline{F}(x)\right)\right)\right\}{\rm d}x.
\end{equation}
We aim to study sufficient conditions on the
distortion functions $h_{\alpha}$ in order to prove the existence of the minimum (or the maximum) for 
$\eta_X(\alpha)$. 
Along the same line, we obtain sufficient conditions on the distortion functions $h_{\alpha}$ and the survival copulas $\widehat{C}_\theta$ in order to prove the existence of the minimum (or the maximum) for
$\nu_X(\theta,\alpha)$.
\par
At least when $X$ and $X_\alpha$ have the same support, for a fixed survival copula  $\widehat{C}_\theta$, a false intuition may suggest that the minimum is reached when $X=_{st}X_\alpha$, i.e. $h_{\alpha}(u)=u$, for all $u\in[0,1]$. However, this is not necessarily true, as illustrated in the following example.
\begin{example}\label{ex:1}
	Suppose that the s.f.\ of 
	$X_\alpha$ is obtained  from $X$ through the distortion $h_\alpha(u)=u^\alpha$, for $u\in[0,1]$. If $X$ has exponential distribution with $\mathsf{E}(X)=1$, then 
	\begin{equation}\label{eq:1 intro}
		\eta_X(\alpha)=\frac{1}{\alpha}+\frac{\alpha-1}{\alpha+1},\qquad\alpha>0.
	\end{equation}
	This function reaches minimum in $\alpha=1+\sqrt{2}$, as shown in the left-hand-side of Figure \ref{fig:variability measure, es intro}.
	Moreover, consider the Farlie-Gumbel-Morgenstern family of survival copulas $\widehat{C}_{\theta}(u,v)=uv\left(1+\theta(1-u)(1-v)\right)$, for $u,v\in[0,1]$ and $\theta\in[-1,1]$.
	In these hypothesis, it follows that
	\begin{equation}\label{eq:2 intro}
		\nu_X(\theta,\alpha)=\frac{1}{\alpha}+\frac{\alpha-1}{\alpha+1}-\frac{3\alpha\theta}{2+7\alpha+7\alpha^2+2\alpha^3},\qquad\theta\in[-1,1],\,\alpha>0,
	\end{equation}
	which does not reach minimum in $\alpha=1$, for all $\theta\in[-1,1]$, as shown in right-hand-side of Figure \ref{fig:variability measure, es intro}.
\end{example}
\begin{figure}[t] 
	\begin{center}
		\includegraphics[scale=0.38]{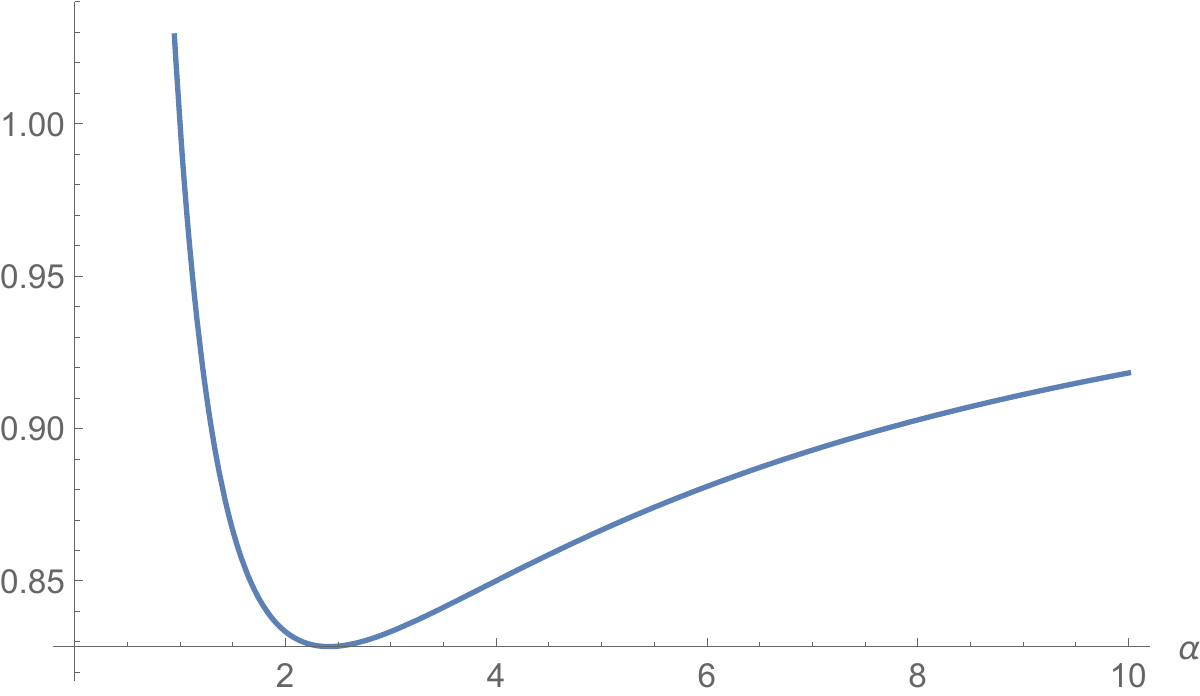}
		\,
		\includegraphics[scale=0.52]{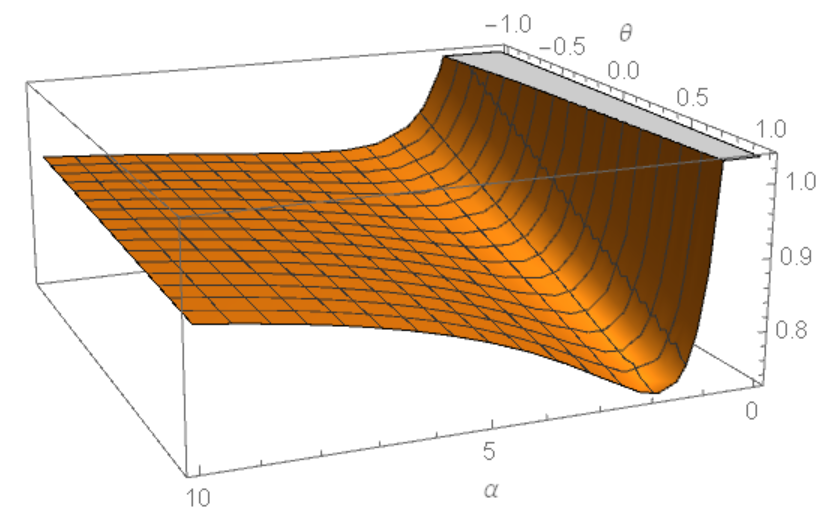}
	\end{center}
	\caption{\small Plots of $\eta_X(\alpha)$ given in Eq.\ (\ref{eq:1 intro}) (left) for $\alpha\in[0,10]$ , and  $\nu_X(\theta,\alpha)$ given in Eq.\ (\ref{eq:2 intro}) (right), for $\theta\in[-1,1],\alpha\in[0,10]$.}
	\label{fig:variability measure, es intro}
\end{figure}
\par
Other examples will be illustrated in the rest of the paper.
\par
	We note that, if there exists a $\theta_I\in\Theta$ such that 
	$$
	\lim_{\theta\to\theta_I}\widehat{C}_{\theta}(u,v)=uv,\qquad\forall u,v\in[0,1],
	$$ 
	then 
	\begin{equation}\label{eq:nu=eta}
		\nu_X(\theta_I,\alpha)=\eta_X(\alpha),\qquad\forall\alpha\in\mathcal{I}.
	\end{equation}
	\begin{remark}\label{rem:variability measure}
	The function defined in Eq.\ (\ref{eq:copula-distorted GMD}) can also be viewed as a variability measure in the sense of Bickel and Lehmann \cite{Bickel:Lehmann}. Indeed, it satisfies the following properties, for all $\theta\in\Theta$ and $\alpha\in\mathcal{I}$:
	\begin{enumerate}
		\item if $X$ and $Y$ have the same distribution, then $\nu_{X}(\theta,\alpha)=\nu_{Y}(\theta,\alpha)$ (law invariance);
		\item $\nu_{X+\delta}(\theta,\alpha)=\nu_{X}(\theta,\alpha)$ for all $X$ and for all constant $\delta\in \mathbb R$ (translation invariance);
		\item $\nu_{\delta X}(\theta,\alpha)=\delta\nu_{X}(\theta,\alpha)$ for all $X$ and for all $\delta\in \mathbb{R}^+$(positive homogeneity);
		\item $	\nu_X(\theta,\alpha)\ge0$ for any random variable $X$, with $\nu_X(\theta,\alpha)=0$ for any degenerate random variable $X$
		(non-negativity);
		\item $X\le_d Y$ implies $	\nu_X(\theta,\alpha)\le	\nu_Y(\theta,\alpha)$ (consistency with dispersive order).
	\end{enumerate}
	\end{remark}
We remark that the copula-distorted Gini's mean difference in Eq.\ (\ref{eq:copula-distorted GMD}) represents an extension of the  measure  studied in Ortega-Jim\'enez et al.\ \cite{patricia:2021}.
Through the paper we implicitly suppose that $\eta_X(\alpha)$ is continuous in $\alpha\in\mathcal{I}$ and $\nu_X(\theta,\alpha)$ is continuous in both $\theta\in\Theta$ and $\alpha\in\mathcal{I}$.
\par 
It is easy to see that, under suitable family of distortion functions, the right-hand-sides of Eqs.\ (\ref{eq:distorted GMD}) and (\ref{eq:copula-distorted GMD}) can be expressed in terms of c.d.f.'s instead of s.f.'s, and copula distributions instead of survival copulas. 
\par 
For different extended versions of Gini's mean difference see for instance
Bassan et al.\ \cite{Bassan:etal}, 
Bernard and M\"uller \cite{bernard:muller}, Capaldo et al.\ \cite{Capaldo:etal} and Vila et al.\ \cite{vila:2023}. 

\subsection{Plan of the paper}
In Section \ref{sec: useful notions} we define some useful notions for our aims  
and provide an example of application arising 
in reliability theory. In Section \ref{sec:independence} we obtain sufficient conditions in order to prove the existence of the minimum or the maximum for the
\emph{distorted Gini's mean difference} $\eta_{X}(\alpha)$.
In Section \ref{sec:main results} we face the same problem for the \emph{copula-distorted Gini's mean difference} $\nu_{X}(\theta,\alpha)$.
Some final remarks are then given in Section \ref{sec: conclusions}. 
	
\section{Basic notions and results}\label{sec: useful notions}
The measures defined in Eqs.\ (\ref{eq:distorted GMD}) and (\ref{eq:copula-distorted GMD}) can be applied also in contexts of reliability theory, where a nonnegative random variable $X$ describes the lifetime of a device or a system. Therefore, a random lifetime will refer to a nonnegative random variable. 
\begin{definition}\label{def:NBU;NWU}
	A random lifetime $X$ with s.f.\ $\overline{F}(x)$ is said to have the New Better than Used (NBU) property if $\overline{F}(x+t)\le\,\overline{F}(x)\overline{F}(t)$ for all $x,t\ge0$. If the inequality is reversed then $X$ is said to have the New Worse than Used (NWU) property.
\end{definition}
\begin{definition}\label{def:IFR;DFR}
	A random lifetime $X$ with s.f.\ $\overline{F}(t)$ and d.f.\ $f(t)$ is said to have the increasing
	failure rate (IFR) property if its hazard rate (h.r.) $\lambda(t)\coloneqq\frac{f(t)}{\overline{F}(t)}$ is increasing in $t$, for all $t$ such that $\overline{F}(t)>0$. If the h.r.\ is decreasing in $t$, then $X$ is said to have the decreasing failure rate (DFR) property.
\end{definition}
For more details about aging notions see for instance Section 4.1 in \cite{navarro:2022}.
\par 
In this paper, for the family of distortion functions $h_\alpha\in\mathcal{H}_\alpha$ involved in $\mathcal{F}_{X,h_\alpha}$ (see Eqs.\ (\ref{eq:family distortions}) and (\ref{eq:family})), we will consider the following assumptions, where ${\rm cl}(\mathcal{I})$ denotes the closure of $\mathcal{I}$:
\begin{description}
	\item[Assumption 2.1] There exists a value $\alpha_I\in{\rm cl}(\mathcal{I})$ such that 
	$$
	\lim_{\alpha\to\alpha_I}h_{\alpha}(u)=u,\qquad\forall u\in[0,1].
	$$
	\item[Assumption 2.2] 
	There exists a value $\alpha_0\in{\rm cl}(\mathcal{I})$ such that 
	$$
	\lim_{\alpha\to\alpha_0}h_{\alpha}(u)=0,\qquad\forall u\in[0,1).
	$$	
\end{description}
\begin{remark}\label{rem:GMD,E}
	It easy to see that
	\begin{description}
		\item[\rm (a)] under the Assumption 2.1, 
		one has $\displaystyle\lim_{\alpha\to\alpha_I}\eta_X(\alpha)={\mathsf{GMD}(X)}$;
		\item[\rm (b)] under the Assumption 2.2, it follows that $\displaystyle	\lim_{\alpha\to\alpha_0}\nu_X(\theta,\alpha)=\mathsf{E}(X)-l$, for all $\theta\in\Theta.$
	\end{description}
\end{remark}
\par
In various contexts of interest the quantile functions are used as an alternative to distribution functions. 
For a random variable $X$ with c.d.f.\ $F(x)$ and d.f.\ $f(x)$, the quantile-density function of $X$ is defined as
$$
q(u)\coloneqq\frac{1}{f\left(F^{-1}(u)\right)}=\frac{{\rm d}}{{\rm d}u}F^{-1}(u),\qquad u\in[0,1]
$$
(see, for instance, Sunoj and Sankaran \cite{Sunoj:2012}). We recall that 
$f\left(F^{-1}(u)\right)$ is called density-quantile function, for $u\in[0,1]$.
Referring to s.f.'s instead of c.d.f.'s, in the following we define a dual version of the quantile-density function. 
\begin{definition}\label{def:dual quantile density}
	Let $X$ be a random variable with s.f.\ $\overline{F}(x)$ and d.f.\ $f(x)$, with support in $(l,r)$. The dual quantile-density function (d.q.d.f.) of $X$ is defined as
	\begin{equation}\label{eq:dual quantile density}
		\tilde{q}(u)=\frac{1}{f\left(\overline{F}^{-1}(u)\right)}=-\frac{\rm{d}}{{\rm d}u}\overline{F}^{-1}(u),\qquad u\in[0,1].
	\end{equation}
\end{definition}
It is easy to see that, from Eq.\ (\ref{eq:dual quantile density}), the d.q.d.f.\ $\tilde{q}(u)$ uniquely identifies the following location-family associated to $\overline{F}$:
$$
\mathcal{\overline{F}}=\left\{\overline{G}: \overline{G}(x)=\overline{F}(x+k),\,\forall x\in\mathbb{R},\, k\in\mathbb{R}\right\}.
$$
By means of the d.q.d.f.\ $\tilde{q}(u)$ we can derive a more tractable expression of the Eqs.\ (\ref{eq:distorted GMD}) and (\ref{eq:copula-distorted GMD}).
\begin{remark}\label{rem:variability measure, dqdf}
	By setting $u=\overline{F}(x)$,  Eqs.\ (\ref{eq:distorted GMD}) and (\ref{eq:copula-distorted GMD}) can be expressed, respectively, as
	 \begin{equation}\label{variability measure, d.q.d.f., independence}
	 	\eta_X(\alpha)=\int_{0}^{1}\tilde{q}(u)\left\{u+h_\alpha(u)\left(1-2u\right)\right\}{\rm d}u,\qquad\alpha\in\mathcal{I},
	 \end{equation}
	 and
	\begin{equation}\label{variability measure, d.q.d.f.}
		\nu_X(\theta,\alpha)=\int_{0}^{1}\tilde{q}(u)\left\{u+h_\alpha(u)-2\,\widehat{C}_\theta\left(u,h_\alpha(u)\right)\right\}{\rm{d}}u,\qquad\theta\in\Theta,\,\alpha\in\mathcal{I}.
	\end{equation}
\end{remark}
Eqs.\ (\ref{variability measure, d.q.d.f., independence}) and (\ref{variability measure, d.q.d.f.})
allow to compare efficiently random variables with different supports, by making use of the following notions: the d.q.d.f.\ $\tilde{q}(u)$, the distortion functions in the family $\mathcal{H}_\alpha$, and also for the Eq.\ (\ref{variability measure, d.q.d.f.}) the associated survival copulas in the family $\mathcal{\hat{C}}^\theta$.
\par
In the main results of the paper, we often make use of the following assumptions regarding $\tilde{q}(u)$.
\begin{description}
	\item[Assumption 2.3] One has 
	$$
	\tilde{q}(u)-\tilde{q}(1-u)\ge0,\qquad\forall u\in\left[0,\frac12\right].
	$$
	\item[Assumption 2.4] One has
	$$
	\tilde{q}(u)-\tilde{q}(1-u)\le0,\qquad\forall u\in\left[0,\frac12\right].
	$$
\end{description}
\begin{proposition}\label{lemma:IFR,DFR and dual quantile density} Let $X$ be a random variable having d.q.d.f. $\tilde{q}(u)$, for $u\in[0,1]$.
	\\
	(i) If $X$ is DFR, then Assumption 2.3 holds. 
	\\
	(ii) If Assumption 2.4 holds, then $X$ is IFR.
\end{proposition}
\begin{proof}
	The proof is an immediate consequence of Definitions \ref{def:IFR;DFR} and \ref{def:dual quantile density}.
\end{proof}
%
\subsection{Hazard models and applications}\label{subsec: hazard models and applications}
Let us now consider some hazard models. 
We recall that the reversed h.r.\ of a random lifetime $X$ is defined has $\tau(x)=f(x)/F(x)$, for $x\in(l,r)$.
\begin{definition}\label{def:distortion functions}
	Let $X$ be a random lifetime with support in $(l,r)$, s.f.\ $\overline{F}(x)$ and h.r.\ $\lambda(x)$.  
	\begin{description}
		\item[\rm (i)] The proportional hazard model is expressed by a random lifetime $X_\alpha$ having h.r.\
		$$
		\lambda_\alpha(x)=\alpha\lambda(x),\qquad\alpha>0,\,x\in(l,r).
		$$
		\item[\rm (ii)] The proportional reversed hazard model is expressed by a random lifetime $X_\alpha$  
		having reversed h.r.\ 
		$$
		\tau_\alpha(x)=\alpha\tau(x),\qquad\alpha>0,\,x\in(l,r),
		$$
		and having h.r.\
		$$
		\lambda_\alpha(x)=\alpha\lambda(x)g(x),\qquad\alpha>0,\,x\in(l,r),
		$$
		where
		\begin{equation}\label{eq:g(x)}
			g_\alpha(x)=\frac{\left[1-\overline{F}(x)\right]^{\alpha-1}-\left[1-\overline{F}(x)\right]^{\alpha}}{1-\left[1-\overline{F}(x)\right]^{\alpha}},\qquad\alpha>0,\,x\in(l,r).
		\end{equation}
		\item[\rm (iii)]
		The generalized additive hazard model is defined by a random lifetime $X_\alpha$ having h.r.\
		$$
		\lambda_\alpha(x)=\lambda(x)+\alpha k(x),\qquad\alpha>0,\,x\in(l,r),
		$$
		where $k(x)$ is a suitable function such that $\lambda_\alpha(x)\ge0$ for all $x\in(l,r)$. 
		\item[\rm (iv)]
		The power hazard model is defined by a random lifetime $X_\alpha$ having h.r.\ 
		$$
		\lambda_\alpha(x)=\alpha\lambda(x^{\alpha})x^{\alpha-1},\qquad\alpha>0,\,x\in(l,r).
		$$ 
	\end{description}
\end{definition}
%
\begin{table}[h]
	\caption{\small The hazard rates and the distortion functions for
		the hazard models introduced in Definition \ref{def:distortion functions}, with $g_\alpha(x)$ defined in Eq.\ (\ref{eq:g(x)}), 	
		$x\in(l,r),\,\alpha>0$ and $u\in[0,1]$.
		In the case {\rm (iii)} we set $K(t)=\int_{0}^{t}k(x){\rm d}x$, for $t>0$. \label{distortion functions}}
	\scalebox{0.85}{
		\begin{tabular}{l@{\hspace{0,5cm}}l@{\hspace{0,5cm}}l}
			\toprule
			hazard models & hazard rates & distortion functions \\
			\midrule
			{\rm (i)} proportional hazard model & $\lambda_\alpha(x)=\alpha\lambda(x)$ & $h_\alpha(u)=u^\alpha$\\
			{\rm (ii)} proportional reversed hazard model & $\lambda_\alpha(x)=\alpha\lambda(x)g_\alpha(x)$ & $h_\alpha(u)=1-(1-u)^\alpha$\\
			{\rm (iii)} generalized additive hazard model & $\lambda_\alpha(x)=\lambda(x)+\alpha k(x)$  & $h_\alpha(u)=u\exp\left\{-\alpha K\left(\overline{F}^{-1}(u)\right)\right\}$ \\
			{\rm (iv)} power hazard model & $\lambda_\alpha(x)=\alpha\lambda(x^{\alpha})x^{\alpha-1}$   & $h_\alpha(u)=\overline{F}\left(\left(\overline{F}^{-1}(u)\right)^\alpha\right)$\\
			\bottomrule
		\end{tabular}
	}
\end{table}
\par
For the proportional hazard model {\rm (i)} see, for instance, Kumar and Klefsj\"o \cite{Kumar:1994}. 
For the proportional reversed hazard model {\rm (ii)} see Di Crescenzo \cite{dicre:2000}, Gupta and Gupta \cite{Gupta:2007}.
We immediately note that, if $k(x)=1$ for all $x$, the model {\rm (iii)} corresponds to the additive hazard model (see Bebbington et al.\ \cite{Bebbington:etal}).
In Table \ref{distortion functions} we show the distortion functions corresponding to the hazard models considered in Definition \ref{def:distortion functions}. 
For such distortion functions one has that Assumptions 2.1 and 2.2 are both satisfied.
\par 
	In reliability theory some interesting problems involve systems consisting of $n$ components. 
	Let $X_1,X_2,\dots,X_n$ describe the independent and identically distributed random lifetimes of each component, with s.f.\ $\overline{F}$.
	In the case of components connected in series,  the system fails as soon as a component stops working. Therefore its lifetime is
	$X_{(1:n)}=\min\{X_1,X_2,\dots,X_n\}$, having s.f.\  
	\begin{equation}\label{eq:s.f. min}
		\overline{F}_{(1:n)}(x)=\left(\overline{F}(x)\right)^n.
	\end{equation}
	When the components are connected in parallel, the system continues to work until the last component fails. So its lifetime is 
	$X_{(n:n)}=\max\{X_1,X_2,\dots,X_n\}$, having s.f.\
	\begin{equation}\label{eq:s.f. max}
		\overline{F}_{(n:n)}(x)=1-\left(1-\overline{F}(x)\right)^n.
	\end{equation}
	We remark that Eq.\ (\ref{eq:s.f. min}) represents the s.f.\ of $X_\alpha$ when $h_\alpha\in\mathcal{H}_\alpha$ comes 
	from the proportional hazard model {\rm (i)} in Table \ref{distortion functions}, for $\alpha=n$.
	On the other hand, Eq.\ (\ref{eq:s.f. max}) represents the s.f.\ of $X_\alpha$ when $h_\alpha\in\mathcal{H}_\alpha$ comes 
	from the proportional reversed hazard model {\rm (ii)} in Table \ref{distortion functions}, for $\alpha=n$.
	\par
	In Example \ref{ex:1}, under the model  {\rm (i)}, we have shown that the minimum of Eq.\ (\ref{eq:1 intro}) is not reached in $\alpha=1$, for $X$ exponentially distributed with $\mathsf{E}(X)=1$. In these assumptions, it is easy to see that $\eta_X(2)<\eta_X(1)$. Therefore, if $X$ describes the lifetime of a single component, in order to reduce the distance, in the sense of Eq.\ (\ref{eq:distorted GMD}), between $X$ and the lifetime of another independent item, it is better to consider a series systems with two independent copies of $X$ instead of another single component distributed as $X$.

\section{Results under independence}\label{sec:independence}
In this section we provide sufficient conditions for the existence of the minimum (or maximum) for the \emph{distorted Gini's mean difference} defined in Eq.\ (\ref{eq:distorted GMD}).
A sketch about the main results of Section \ref{sec:independence} is provided in Table \ref{tab:main results, independence}.
	\begin{table}[h]
		\caption{\small A sketch about the main results of Section \ref{sec:independence}, with indications of the involved quantities. 	\label{tab:main results, independence}}
		\scalebox{0.80}{
			\begin{tabular}{l@{\hspace{0,5cm}}l@{\hspace{0,5cm}}l}
				\toprule
				Theorems & Relevant topics & Results \\
				\midrule
				Theorem \ref{prop:v0<GMD} & $\tilde{q}(u),h_\alpha(u)$ & $\eta_X(\alpha)\le(\ge)\,\mathsf{GMD}(X)$\\
				Theorem \ref{prop:v0<E}& $\tilde{q}(u), h_\alpha(u)$ & $\eta_X(\alpha)\le(\ge)\,\mathsf{E}(X)-l$	\\
				Theorem \ref{prop:derivata, independence} & $\tilde{q}(u),\frac{\partial}{\partial\alpha}h_\alpha(u)$ & $\displaystyle\lim_{\alpha\to\alpha_I}
				\frac{\rm d}{{\rm d}\alpha}\eta_X(\alpha)\le(\ge)\,0$\\
				Theorem \ref{thm:NBU(NWU)} & $h_\alpha(u)=u\exp\left\{-\alpha K\left(\overline{F}^{-1}(u)\right)\right\}$ with $K(\cdot)$ increasing (decreasing), & a local minimum (maximum) \\	
				& NBU (NWU) property, $\mathsf{G}(X)$ &	for $\eta_X(\alpha)$ 
				\\
				\bottomrule
			\end{tabular}
		}
	\end{table}
	\par
	In this section we assume that	$X$ is a random variable 	with s.f.\ $\overline{F}(x)$, d.f.\ $f(x)$, for $x\in\mathbb{R}$, and d.q.d.f.\ $\tilde{q}(u)$, for $u\in[0,1]$, 
	having support in $(l,r)$ and having finite expected value $\mathsf{E}(X)$.
	Moreover, we assume that $X_\alpha\in\mathcal{F}_{X,h_\alpha}$. 
	\begin{theorem}\label{prop:v0<GMD}
		Let $(X, X_\alpha)$ be a random vector with independent components.
		For all $\alpha\in\mathcal{I}$ such that  
		\begin{equation}\label{cond1:prop1,independence}
			\tilde{q}(1-u)\left[1-u-h_{\alpha}(1-u)\right]\le(\ge)\,\tilde{q}(u)\left[u-h_{\alpha}(u)\right],\qquad\forall u\in\left[0,\frac12\right],
		\end{equation}
		then
		\begin{equation}\label{eq:th1:ind}
			\eta_X(\alpha)\le(\ge)\,\mathsf{GMD}(X).
		\end{equation}
	\end{theorem}
	\begin{proof}
	 Making use of Eq.\ (\ref{variability measure, d.q.d.f., independence}), with few calculations one has
		$$
		\mathsf{GMD}(X)-\eta_X(\alpha)=\int_{0}^{\frac12}(1-2u)\left\{\tilde{q}(u)\left[u-h_{\alpha}(u)\right]-\tilde{q}(1-u)\left[1-u-h_{\alpha}(1-u)\right]\right\}{\rm d}u.
		$$		
		Hence, Eq.\ (\ref{eq:th1:ind}) follows from Eq.\ (\ref{cond1:prop1,independence}).
	\end{proof}
	The next result follows from Remark \ref{rem:GMD,E}(a), recalling that $\alpha_I$ is defined as in Assumption 2.1.
	\begin{corollary}
	Under the Assumption 2.1 and the hypothesis of Theorem \ref{prop:v0<GMD}, for all $\alpha\in\mathcal{I}$ such that Eq.\ (\ref{cond1:prop1,independence}) holds, one has 
	\begin{equation}\label{eq:cor1}
	\eta_X(\alpha)\le(\ge)\,\eta_X(\alpha_I).
	\end{equation}
	\end{corollary}
	\begin{theorem}\label{prop:v0<E}
		Let $(X, X_\alpha)$ be a random vector with independent components.
		For all $\alpha\in\mathcal{I}$ such that 
		\begin{equation}\label{cond1:prop2,independence}
			\tilde{q}(u)h_{\alpha}(u)\le(\ge)\,\tilde{q}(1-u)h_{\alpha}(1-u),\qquad\forall u\in\left[0,\frac12\right],
		\end{equation}
		then 
		\begin{equation}\label{eq:th2:ind}
			\eta_X(\alpha)\le(\ge)\,\mathsf{E}(X)-l.
		\end{equation}
	\end{theorem}
	\begin{proof}
		Making use of Eq.\ (\ref{variability measure, d.q.d.f., independence}), with few calculations one has 
		$$
		\mathsf{E}(X)-l-\eta_X(\alpha)=\int_{0}^{\frac12}(1-2u)\left\{\tilde{q}(1-u)h_{\alpha}(1-u)-\tilde{q}(u)h_{\alpha}(u)\right\}{\rm d}u.
		$$
		Therefore, the thesis follows from hypothesis (\ref{cond1:prop2,independence}).	
	\end{proof}
	The next result follows from Remark \ref{rem:GMD,E}{\rm (b)}, recalling that $\alpha_0$ is defined as in Assumption 2.2.
	\begin{corollary}
		Under the Assumption 2.2 and the hypothesis of Theorem \ref{prop:v0<E}, for all $\alpha\in\mathcal{I}$ such that Eq.\ (\ref{cond1:prop2,independence}) holds, one has 
		\begin{equation}\label{eq:cor2}
		\eta_X(\alpha)\le(\ge)\,\eta_X(\alpha_0).
		\end{equation}
	\end{corollary}
	When both Eqs.\ (\ref{eq:cor1}) and (\ref{eq:cor2}) hold for
	$\alpha\in\left(\min\{\alpha_I,\alpha_0\},\max\{\alpha_I,\alpha_0\}\right)\subseteq{\rm cl}(\mathcal{I})$, then there exists at least an $\alpha'\in\left(\min\{\alpha_I,\alpha_0\},\max\{\alpha_I,\alpha_0\}\right)$ such that $\eta_X(\alpha)$ attains a local minimum (maximum) in $\alpha=\alpha'$.
	\begin{remark}
		Let $X$ be uniformly distributed over $[l,r]$, with $l,r\in\mathbb{R}$. 
		In this case Eq.\ (\ref{cond1:prop2,independence}) is valid for all $h_\alpha\in\mathcal{H}_\alpha$, so that from Theorem \ref{prop:v0<E} one has
		$$
		\eta_X(\alpha)\le\frac{r-l}{2},\qquad\forall \alpha\in\mathcal{I}.
		$$
		In addition, we remark that under the models {\bf (i)}, {\bf (ii)} and {\bf (iv)} of Table \ref{distortion functions}, one has
		$$
		\eta_X(\alpha)=\frac{1}{2}+\frac{1}{1+\alpha}-\frac{2}{2+\alpha},\qquad\alpha>0.
		$$
	\end{remark}
	\begin{corollary}\label{cor:GMD,E}
	Let $(X, X_\alpha)$ be a random vector with independent components. If Assumption 2.4 (Assumption 2.3) hold, then
		\begin{equation}\label{eq:GMD,E(X)}
			\mathsf{GMD}(X)\le(\ge)\,\mathsf{E}(X)-l.
		\end{equation}
	\end{corollary}
	\begin{proof}
		The thesis follows from Eq.\  (\ref{eq:th2:ind}) of Theorem \ref{prop:v0<E}, by noting that Assumption 2.4 (Assumption 2.3) represent sufficient conditions for the Eq.\ (\ref{cond1:prop2,independence}) for $\alpha=\alpha_I$. 
	\end{proof}
    In the next theorem we aim to obtain sufficient conditions in order to specify the sign of ${\displaystyle\lim_{\alpha\to\alpha_I}}\frac{\rm d}{{\rm d}\alpha}\eta_X(\alpha)$, where due to Eq.\ (\ref{variability measure, d.q.d.f., independence}) and by differentiation under the integral sign, one has
    \begin{equation}\label{eq:derivative variability measure, independence}
    	\frac{\rm d}{{\rm d}\alpha}\eta_X(\alpha)=\int_{0}^{1}\tilde{q}(u)\frac{\partial}{\partial\alpha}h_\alpha(u)\left(1-2u\right){\rm d}u,\qquad\alpha\in\mathcal{I}.
    \end{equation}
    \begin{theorem}\label{prop:derivata, independence}
    	Let $(X, X_\alpha)$ be a random vector with independent components.
    	If
    	\begin{equation}\label{cond1:thm3}
    		\lim_{\alpha\to\alpha_I}\tilde{q}(u)\frac{\partial}{\partial\alpha}h_{\alpha}(u)\le(\ge)\,\lim_{\alpha\to\alpha_I}\tilde{q}(1-u)\frac{\partial}{\partial\alpha}h_{\alpha}(1-u),\qquad\forall u\in\left[0,\frac12\right],
    	\end{equation}
    	then
    	\begin{equation}\label{eq:thm3}
    		\lim_{\alpha\to\alpha_I}
    		\frac{\rm d}{{\rm d}\alpha}\eta_X(\alpha)\le(\ge)\,0.
    	\end{equation} 
    \end{theorem}
    \begin{proof}
    	Making use of Eq.\ (\ref{variability measure, d.q.d.f., independence}), from Eq.\ (\ref{eq:derivative variability measure, independence}), with some calculations one has 
    	$$
    	\lim_{\alpha\to\alpha_I}
    	\frac{\rm d}{{\rm d}\alpha}\eta_X(\alpha)=\int_{0}^{\frac12}\left(1-2u\right)\left\{\lim_{\alpha\to\alpha_I}\left[\tilde{q}(u)\frac{\partial}{\partial\alpha}h_{\alpha}(u)-\tilde{q}(1-u)\frac{\partial}{\partial\alpha}h_{\alpha}(1-u)\right]\right\}{\rm d}u.
    	$$
    	Therefore, Eq.\ (\ref{eq:thm3}) follows from Eq.\ (\ref{cond1:thm3})
    	\end{proof}
	\par
	Recalling Assumptions 2.1 and 2.2, if Eqs.\ (\ref{eq:GMD,E(X)}) and (\ref{eq:thm3}) hold, then there exists a local minimum (maximum) of $\eta_X(\alpha)$, due to the continuity of $\eta_X$.
	\begin{example}
		Let $X$ be uniformly distributed on $[0,1]$. Let $X_\alpha\in\mathcal{F}_{X,h_\alpha}$, with distorted s.f.\ obtained through the generalized additive hazard model {\rm (iii)} in Table \ref{distortion functions}, for $K(t)=\frac{t^2}{2}$. Since $\overline{F}^{-1}(u)=1-u$, with $u\in[0,1]$, one has 
		$$
		h_\alpha(u)=u\exp\left\{{-\alpha\frac{(1-u)^2}{2}}\right\},\qquad u\in[0,1].
		$$
		Therefore, Eq.\ (\ref{variability measure, d.q.d.f., independence}) becomes
		\begin{equation}\label{variability measure, independence, es 1}
			\eta_X(\alpha)=\frac{6-2e^{-\alpha/2}+\alpha}{2\alpha}-\frac{\sqrt{\frac{\pi}{2}}(2+\alpha)\mathsf{Erf}(\frac{\sqrt{\alpha}}{\sqrt{2}})}{\alpha^\frac32},\qquad\alpha>0,
		\end{equation}
		where $\mathsf{Erf}(x)=\frac{2}{\sqrt{\pi}}\int_{0}^{x}\exp\left\{-t^2\right\}\text{d}t$, for $x>0$, is the well known error function.
		Since $h_\alpha$ satisfies Eqs.\ (\ref{cond1:prop1,independence}) and (\ref{cond1:prop2,independence}), there exists the minimum for the $\eta_X(\alpha)$ given in Eq.\ (\ref{variability measure, independence, es 1}). This is confirmed by the plot in the left-hand-side of Figure \ref{fig:variability measure, es 4.1,4.2}.
	\end{example}
	\begin{figure}[t] 
		\begin{center}
			\includegraphics[scale=0.38]{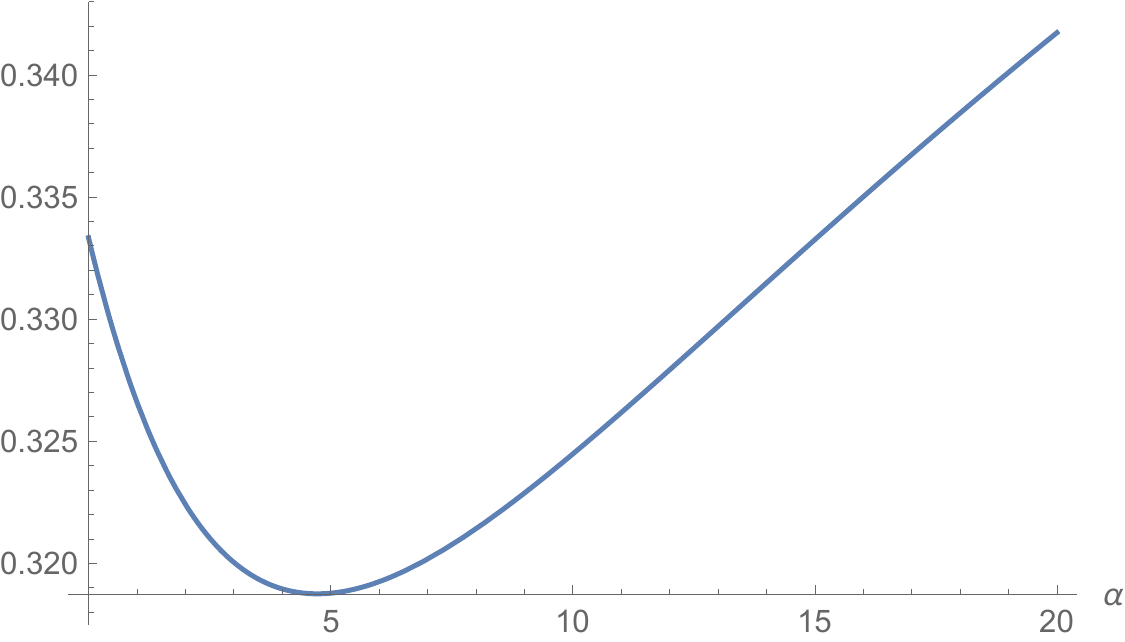}
			\,
			\includegraphics[scale=0.38]{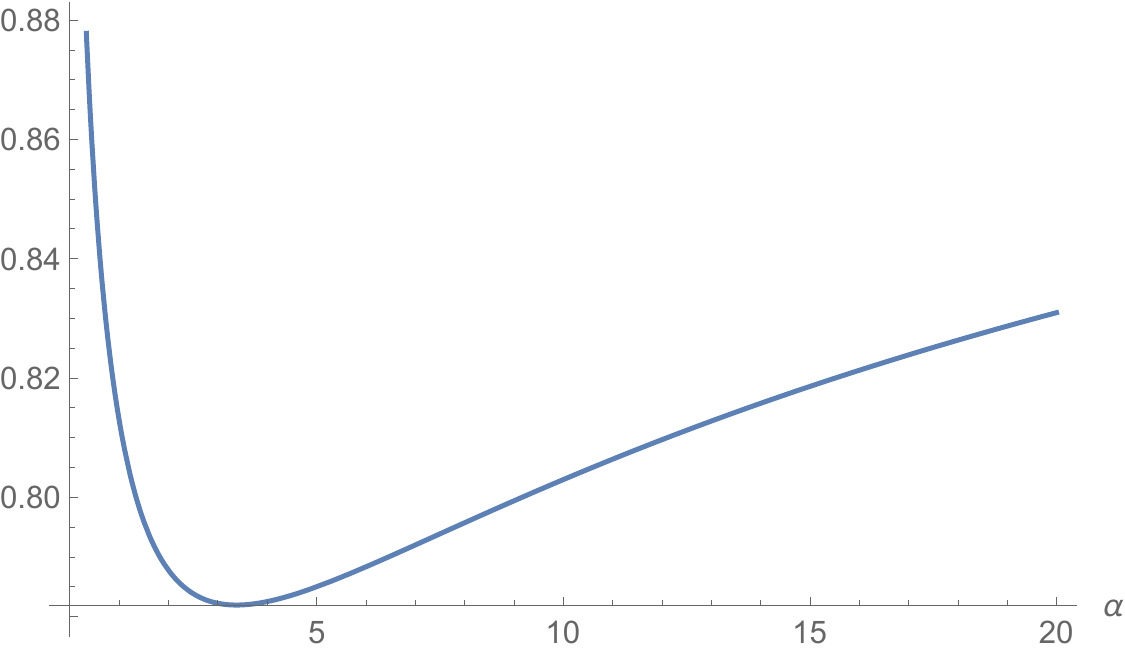}
		\end{center}
		\caption{\small Plots of $\eta_X(\alpha)$ given in Eq.\ (\ref{variability measure, independence, es 1}) (left) and given in Eq.\ (\ref{variability measure, independence, es 2}) (right), for $\alpha\in[0,20]$.}
		\label{fig:variability measure, es 4.1,4.2}
	\end{figure}
	\begin{example}
		Let $X$ be an exponential random variable having $\mathsf{E}(X)=1$. Let $X_\alpha\in\mathcal{F}_{X,h_\alpha}$, with distorted s.f.\ obtained through the generalized additive hazard model {\rm (iii)} in Table \ref{distortion functions}, for $K(t)=\frac{t^2}{2}$. Since $\overline{F}^{-1}(u)=-\log(u)$, with $u\in[0,1]$, one has 
		$$
		h_\alpha(u)=u\exp\left\{{-\alpha\frac{(-\log(u))^2}{2}}\right\},\qquad u\in[0,1].
		$$
		Therefore, Eq.\ (\ref{variability measure, d.q.d.f., independence}) becomes
		\begin{equation}\label{variability measure, independence, es 2}
			\eta_X(\alpha)=1+\frac{e^{\frac{1}{2\alpha}}\sqrt{\frac{\pi}{2}}\mathsf{Erfc}(\frac{1}{\sqrt{2\alpha}})}{\sqrt{\alpha}}-\frac{e^{\frac2\alpha}\sqrt{2\pi}\mathsf{Erfc}(\frac{\sqrt{2}}{\sqrt{\alpha}})}{\sqrt{\alpha}},\qquad\alpha\ge0.
		\end{equation}
		where $\mathsf{Erfc}(x)=1-\mathsf{Erf}(x)$ is the complementary error function.
		Since $h_\alpha$ satisfies Eqs.\ (\ref{cond1:prop1,independence}) and (\ref{cond1:prop2,independence}), there exists the minimum for the $\eta_X(\alpha)$ given in Eq.\ (\ref{variability measure, independence, es 2}). This is confirmed by the plot in the right-hand-side of Figure \ref{fig:variability measure, es 4.1,4.2}.
	\end{example}
	\par
	Let us now consider the family of distortion functions $\mathcal{H}_\alpha$ under the generalized additive hazard model {\rm (iii)} in Table \ref{distortion functions}. In this hypothesis, Eq.\ (\ref{eq:distorted GMD}) becomes
	\begin{equation}\label{eq:variability measure 3}
		\eta_X(\alpha)=\int_{l}^{r}\overline{F}(x)\left\{1+e^{\left\{-\alpha K\left(x\right)\right\}}\left[1-2\overline{F}(x)\right]\right\}{\rm d}x,
	\end{equation}
	and also  
	\begin{equation}\label{eq:der,additive}
		\frac{\rm d}{{\rm d}\alpha}\eta_X(\alpha)=\int_{l}^{r}K(x)e^{\left\{-\alpha K(x)\right\}}\overline{F}(x)\left[2\overline{F}(x)-1\right]{\rm d}x.
	\end{equation}
	\begin{theorem}\label{thm:NBU(NWU)}
		Let $(X, X_\alpha)$ be a random vector with independent components having support $(0,+\infty)$. Let $\eta_X(\alpha)$ be as in Eq.\ (\ref{eq:variability measure 3}) and let $\mathsf{G}(X)$ be the Gini's index of $X$.
		Suppose that $K(x)$ is increasing (decreasing) in $x>0$. If $X$ is NWU (NBU) and $\mathsf{G}(X)\le(\ge)\,\frac12$, then there exists the minimum (maximum) for
		$\eta_X(\alpha)$.
	\end{theorem}
	\begin{proof}
		If $X$ is NWU (NBU), then 
		$$
			\int_{0}^{+\infty}\overline{F}(x)\left[2\overline{F}(x)-1\right]{\rm d}x\le(\ge)\, 2\int_{0}^{+\infty}\overline{F}(2x){\rm d}x-\mathsf{E}(X)=0.	
		$$	
		Hence, from Eq.\ (\ref{eq:der,additive}), if $K(x)$ is increasing (decreasing) in $x>0$, then
		$$
		\lim_{\alpha\to\alpha_I}
		\frac{\rm d}{{\rm d}\alpha}\eta_X(\alpha)
		=\int_{0}^{+\infty}K(x)\overline{F}(x)\left[2\overline{F}(x)-1\right]{\rm d}x\le(\ge)\,0.
		$$
		The thesis follows recalling that, from Eq.\ (\ref{eq:rmd,gini index}), one has  $\mathsf{G}(X)\le(\ge)\,\frac12$ if and only if $\mathsf{GMD}(X)\le(\ge)\,\mathsf{E}(X)$.
	\end{proof}
	%

\section{Results under dependence}\label{sec:main results}
In Section \ref{sec:independence} we considered sufficient conditions for the existence of the minimum (or the maximum) for the distorted Gini's mean difference. 
In this section we face the same problem under the different hypothesis that $X$ and $X_\alpha$ are dependent according to the survival copula $\widehat{C}_\theta\in\mathcal{\hat{C}}^\theta$, i.e.\ regarding the copula-distorted Gini's mean difference. 
A sketch about the main results of Section \ref{sec:main results} is provided in Table \ref{tab:main results}.
\begin{table}[h]
	\caption{\small A sketch about the main results of Section \ref{sec:main results}, with indications of the involved quantities. In particular, $\partial_2 \widehat{C}_{\theta}(u,u)$ is defined in Eq.\ (\ref{eq:der C}). \label{tab:main results}}
	\scalebox{0.80}{
		\begin{tabular}{l@{\hspace{0,5cm}}l@{\hspace{0,5cm}}l}
			\toprule
			Theorems & Relevant topics & Results \\
			\midrule
			Theorem \ref{prop:v<v0} & $\widehat{C}_\theta\left(u,h_\alpha(u)\right), h_\alpha(u)$, Assumptions 2.1, 2.3, 2.4  & $\nu_X(\theta,\alpha)\le(\ge)\,\nu_X(\theta,\alpha_I)$\\
			Theorem \ref{prop:v<E}&$\widehat{C}_\theta\left(u,h_\alpha(u)\right), h_\alpha(u)$, Assumptions 2.3 and 2.4 &	$\nu_X(\theta,\alpha)\le(\ge)\,\mathsf{E}(X)-l$\\
			Theorem \ref{prop:derivata}& $\partial_2 \widehat{C}_{\theta}(u,u),\tilde{q}(u),\frac{\partial}{\partial\alpha}h_{\alpha}(u)$ & $\displaystyle{\lim_{\alpha\to\alpha_I}}
			\frac{\partial}{\partial\alpha}\nu_X(\theta,\alpha)\ge(\le)\,0$ \\
			Theorem \ref{prop:bound}& $h_\alpha(u),\tilde{q}(u)$  & $\nu_X(\theta,\alpha_I)\le\,\mathsf{E}(X)-l$ \\
			\bottomrule
		\end{tabular}
	}
\end{table}
\par
In this section we assume that	$X$ is a random variable 	with s.f.\ $\overline{F}(x)$, d.f.\ $f(x)$, for $x\in\mathbb{R}$, and d.q.d.f.\ $\tilde{q}(u)$, for $u\in[0,1]$, 
having support in $(l,r)$ and having finite expected value $\mathsf{E}(X)$.
Moreover, we assume that $X_\alpha\in\mathcal{F}_{X,h_\alpha}$. 
\begin{theorem}\label{prop:v<v0}
	Let $(X,X_\alpha)$ be the random vector with survival copula $\widehat{C}_\theta\in\mathcal{\hat{C}}^\theta$.
	Under the Assumption 2.1, suppose that for all $\alpha\neq\alpha_I$ and for all $\theta\in\Theta$, one has for all  $u\in\left[0,\frac12\right]$ 
	\begin{equation}\label{cond1:prop1}
		\widehat{C}_{\theta}(u,h_{\alpha}(u))+\widehat{C}_{\theta}(1-u,h_{\alpha}(1-u))-\left[\widehat{C}_{\theta}(u,u)+\widehat{C}_{\theta}(1-u,1-u)\right]-
		\frac12[h_{\alpha}(u)+h_{\alpha}(1-u)-1]=\delta_\theta(u)
	\end{equation}
	and
	\begin{equation}\label{cond1b:prop1}
	\int_{0}^{\frac12}\tilde{q}(1-u)\delta_\theta(u)\text{d}u\ge(\le)\,0.
	\end{equation}
	Let
	\begin{equation}\label{cond2:prop1}
		\widehat{C}_{\theta}(u,h_{\alpha}(u))-\widehat{C}_{\theta}(u,u)\ge(\le)\,\frac12\left[h_{\alpha}(u)-u\right],\qquad\forall u\in\left[0,\frac12\right],\forall\theta\in\Theta.
	\end{equation}
	If Assumption 2.3 (Assumption 2.4) holds, then 
	\begin{equation}\label{eq:min,prop1}
		\nu_X(\theta,\alpha)\le\nu_X(\theta,\alpha_I),\qquad\forall\theta\in\Theta,\forall\alpha\in\mathcal{I}.
	\end{equation} 
	If Assumption 2.4 (Assumption 2.3) holds, then 
	\begin{equation}\label{eq:max,prop1}
		\nu_X(\theta,\alpha)\ge\nu_X(\theta,\alpha_I),\qquad\forall\theta\in\Theta,\forall\alpha\in\mathcal{I}.
	\end{equation} 
\end{theorem}
\begin{proof}
	From Assumption 2.1 and Eqs.\ (\ref{variability measure, d.q.d.f.}) and (\ref{cond1:prop1}), with few calculations one has 
	$$
	\nu_X(\theta,\alpha_I)-\nu_X(\theta,\alpha)= \int_{0}^{\frac12}s(u) \left[\tilde{q}(u)-\tilde{q}(1-u)\right]{\rm d}u+\int_{0}^{\frac12}\tilde{q}(1-u)\delta_\theta(u)\text{d}u,\qquad\forall\theta\in\Theta,
	$$
	where we have set $s(u)\coloneqq u-h_{\alpha}(u)+2\left[\widehat{C}_{\theta}(u,h_{\alpha}(u))-\widehat{C}_{\theta}(u,u)\right]$, for $u\in\left[0,\frac12\right]$. 
	Making use of Eqs.\ (\ref{cond1b:prop1}) and (\ref{cond2:prop1}), Eq.\ (\ref{eq:min,prop1}) follows from Assumption 2.3 (Assumption 2.4). In the same way Eq.\ (\ref{eq:max,prop1}) can be obtained.
\end{proof}
\begin{theorem}\label{prop:v<E}
	Let $(X,X_\alpha)$ be the random vector with survival copula $\widehat{C}_\theta\in\mathcal{\hat{C}}^\theta$.
	Suppose that for all $\theta\in\Theta$ and for all $\alpha\in\mathcal{I}$ one has
	\begin{equation}\label{cond1:prop2}
		\widehat{C}_{\theta}(u,h_{\alpha}(u))+\widehat{C}_{\theta}(1-u,h_{\alpha}(1-u))-\frac12[h_{\alpha}(u)+h_{\alpha}(1-u)]=\delta_\theta(u),\qquad\forall u\in\left[0,\frac12\right],
	\end{equation}
	and
	\begin{equation}\label{cond1b:prop2}
	\int_{0}^{\frac12}\tilde{q}(1-u)\delta_\theta(u)\text{d}u\ge(\le)\,0.
	\end{equation}
	Let
	\begin{equation}\label{cond2:prop2}
		\widehat{C}_{\theta}(u,h_{\alpha}(u))\ge(\le)\,\frac12h_{\alpha}(u),\qquad\forall u\in\left[0,\frac12\right].
	\end{equation}
	If Assumption 2.3 (Assumption 2.4) holds, then
	\begin{equation}\label{eq:min,prop2}
		\nu_X(\theta,\alpha)\le\mathsf{E}(X)-l,\qquad\forall\theta\in\Theta,\forall\alpha\in\mathcal{I}.
	\end{equation} 
	If Assumption 2.4 (Assumption 2.3) holds, then
	\begin{equation}\label{eq:max,prop2}
		\nu_X(\theta,\alpha)\ge\mathsf{E}(X)-l,\qquad\forall\theta\in\Theta,\forall\alpha\in\mathcal{I}.
	\end{equation} 
\end{theorem}
\begin{proof}
	From Eqs.\ (\ref{variability measure, d.q.d.f.}) and (\ref{cond1:prop2}), after few calculations we obtain	
	$$
	\mathsf{E}(X)-l-\nu_X(\theta,\alpha)=\int_{0}^{\frac12}b(u)\left[\tilde{q}(u)-\tilde{q}(1-u)\right]{\rm d}u+\int_{0}^{\frac12}\tilde{q}(1-u)\delta_\theta(u)\text{d}u,\qquad\forall\theta\in\Theta,
	$$
	where we have set $b(u)\coloneqq2\widehat{C}_{\theta}(u,h_{\alpha}(u))-h_{\alpha}(u)$, for $u\in\left[0,\frac12\right]$. 
	Making use of Eqs.\ (\ref{cond1b:prop2}) and (\ref{cond2:prop2}), 
	Eq.\ (\ref{eq:min,prop2}) follows from Assumption 2.3 (Assumption 2.4). In a similar way Eq.\ (\ref{eq:max,prop2}) can be obtained.
\end{proof}
The next result follows from Remark \ref{rem:GMD,E}{\rm (b)}, recalling that $\alpha_0$ is defined as in Assumption 2.2.
\begin{corollary}\label{cor:4.1}
	Under the Assumption 2.2, Eqs.\ (\ref{eq:min,prop2}) and (\ref{eq:max,prop2}) become respectively 
	\begin{equation}\label{eq:1cor4.1}
		\nu_X(\theta,\alpha)\le\nu_X(\theta,\alpha_0),\qquad\forall\theta\in\Theta,\forall\alpha\in\mathcal{I},
	\end{equation}
	\begin{equation}\label{eq:2cor4.1}
		\nu_X(\theta,\alpha)\ge\nu_X(\theta,\alpha_0),\qquad\forall\theta\in\Theta,\forall\alpha\in\mathcal{I}.
	\end{equation}
	In addition, under the Assumption 2.1, one has 
	\begin{equation}\label{eq:min,prop2, cor}
		\nu_X(\theta,\alpha_I)\le\,\nu_X(\theta,\alpha_0),\qquad\forall\theta\in\Theta,
	\end{equation}
	\begin{equation}\label{eq:max,prop2, cor}
		\nu_X(\theta,\alpha_I)\ge\,\nu_X(\theta,\alpha_0),\qquad\forall\theta\in\Theta.
	\end{equation}
\end{corollary}
As in the previous section, when both Eqs.\ (\ref{eq:min,prop1}) and (\ref{eq:1cor4.1}) hold for
$\alpha\in\left(\min\{\alpha_I,\alpha_0\},\max\{\alpha_I,\alpha_0\}\right)$, it means that for any fixed $\theta\in\Theta$, there exists at least an $\alpha'\in\mathcal{I}$ such that $\nu_X(\theta,\alpha)$ attains a local minimum in $\alpha=\alpha'$. Similarly, when both Eqs.\ (\ref{eq:max,prop1}) and (\ref{eq:2cor4.1}) hold for $\alpha\in\left(\min\{\alpha_I,\alpha_0\},\max\{\alpha_I,\alpha_0\}\right)$, there exists a local maximum of $\nu_X(\theta,\alpha)$ in $\alpha=\alpha'$ for any fixed $\theta\in\Theta$.
\begin{remark}\label{remark:uniformly}
If $X$ is uniformly distributed over $[l,r]$, $l,r\in\mathbb{R}$, the d.q.d.f.\ is given by  $\tilde{q}(u)=r-l>0$, $u\in[0,1]$. Therefore, one can show that the thesis of Theorems \ref{prop:v<v0} and \ref{prop:v<E} and Corollary \ref{cor:4.1} follow with a similar procedure without using Assumptions 2.3 or 2.4.
\end{remark}
\par
For $u,v\in[0,1]$, we adopt the following notation 
\begin{equation}\label{eq:der C}
	\partial_2 \widehat{C}_{\theta}(u,v)\coloneqq\frac{\partial}{\partial x_2}
	\widehat{C}_{\theta}(x_1,x_2)\Big|_{(x_1,x_2)=(u,v)},\qquad\forall\theta\in\Theta. 
\end{equation}
As in Theorem \ref{prop:derivata, independence}, we now obtain sufficient conditions in order to specify the sign of ${\displaystyle\lim_{\alpha\to\alpha_I}}\frac{\partial}{\partial\alpha}\nu_X(\theta,\alpha)$, where due to Eq.\ (\ref{variability measure, d.q.d.f.}) and by differentiation under the integral sign, one has
\begin{equation}\label{eq:derivative variability measure}
	\frac{\partial}{\partial\alpha}\nu_X(\theta,\alpha)=\int_{0}^{1}\tilde{q}(u)\frac{\partial}{\partial\alpha}h_\alpha(u)\left[1-2\partial_2 \widehat{C}_{\theta}(u,h_\alpha(u))\right]{\rm d}u,\qquad\theta\in\Theta,\,\alpha\in\mathcal{I}.
\end{equation}
\begin{theorem}\label{prop:derivata}
	Let $(X,X_\alpha)$ be the random vector with survival copula $\widehat{C}_\theta\in\mathcal{\hat{C}}^\theta$. 
	Suppose that 
	\begin{equation}\label{cond1:prop3}
		\partial_2\widehat{C}_{\theta}(1-u,1-u)+\partial_2\widehat{C}_{\theta}(u,u)=1,\qquad\forall u\in\left[0,\frac12\right],\forall\theta\in\Theta,
	\end{equation}
	and 
	\begin{equation}\label{cond2:prop3}
		\partial_2\widehat{C}_{\theta}(u,u)\le(\ge)\,\frac12,\qquad\forall u\in\left[0,\frac12\right],\forall\theta\in\Theta.
	\end{equation}
	If
	\begin{equation}\label{cond3:min,prop3}
		\lim_{\alpha\to\alpha_I}\tilde{q}(u)\frac{\partial}{\partial\alpha}h_{\alpha}(u)\le(\ge)\,\lim_{\alpha\to\alpha_I}\tilde{q}(1-u)\frac{\partial}{\partial\alpha}h_{\alpha}(1-u),\qquad\forall u\in\left[0,\frac12\right]
	\end{equation}
	then
	\begin{equation}\label{eq:min,prop3}
		\lim_{\alpha\to\alpha_I}
		\frac{\partial}{\partial\alpha}\nu_X(\theta,\alpha)\le0,\qquad\forall\theta\in\Theta.
	\end{equation} 
	If
	\begin{equation}\label{cond3:max,prop3}
		\lim_{\alpha\to\alpha_I}\tilde{q}(u)\frac{\partial}{\partial\alpha}h_{\alpha}(u)\ge(\le)\,\lim_{\alpha\to\alpha_I}\tilde{q}(1-u)\frac{\partial}{\partial\alpha}h_{\alpha}(1-u),\qquad\forall u\in\left[0,\frac12\right]
	\end{equation}
	then
	\begin{equation}\label{eq:max,prop3}
		\lim_{\alpha\to\alpha_I}
		\frac{\partial}{\partial\alpha}\nu_X(\theta,\alpha)\ge0,\qquad\forall\theta\in\Theta.
	\end{equation} 
\end{theorem}
\begin{proof}
	In the assumption (\ref{cond1:prop3}), making use of Eq.\ (\ref{eq:derivative variability measure}), with few calculations one has 
	$$
	\lim_{\alpha\to\alpha_I}
	\frac{\partial}{\partial\alpha}\nu_X(\theta,\alpha)=\int_{0}^{\frac12}\left[1-2\partial_2\widehat{C}_{\theta}(u,u)\right]\left\{\lim_{\alpha\to\alpha_I}\left[\tilde{q}(u)\frac{\partial}{\partial\alpha}h_{\alpha}(u)-\tilde{q}(1-u)\frac{\partial}{\partial\alpha}h_{\alpha}(1-u)\right]\right\}{\rm d}u.
	$$
	Therefore, Eq.\ (\ref{eq:min,prop3}) follows from Eqs.\ (\ref{cond2:prop3}) and (\ref{cond3:min,prop3}).
	In the same way Eq.\ (\ref{eq:max,prop3}) follows from Eqs.\ (\ref{cond2:prop3}) and (\ref{cond3:max,prop3}).
\end{proof}
It easy to see that the independence copula satisfies Eqs.\ (\ref{cond1:prop3}) and (\ref{cond2:prop3}). Therefore Theorem \ref{prop:derivata} for $\frac{\partial}{\partial\alpha}\nu_X(\theta,\alpha)$  generalizes Theorem \ref{prop:derivata, independence} for $\frac{\rm d}{{\rm d}\alpha}\eta_X(\alpha)$ .
\par
Recalling Assumptions 2.1 and 2.2, if Eqs.\ (\ref{eq:min,prop2, cor}) and (\ref{eq:min,prop3}) hold, then there exists a local minimum of $\nu_X(\theta,\alpha)$, due to the continuity of $\nu_X$, for any fixed $\theta\in\Theta$.
For the same reasons, if Eqs.\ (\ref{eq:max,prop2, cor}) and (\ref{eq:max,prop3}) hold, then there exists the local maximum of $\nu_X(\theta,\alpha)$.
\par
	\begin{lemma}
	For a copula function $C(u,v)$, for $u,v\in[0,1]$, one has
	\begin{equation}\label{eq:bound diagonal}
		C(u,u)\ge \max\{2u-1,0\},\qquad\forall u\in[0,1].
	\end{equation}
	\end{lemma}
	\begin{proof}
	The proof immediately follows when $u=v$ in the left inequality of the following Fr\'echet-Hoeffding bounds 
	$$
		\max\{u+v-1,0\}\le C(u,v)\le\min\{u,v\},  
	$$
	that hold for every copula $C(u,v)$ and for all $u,v\in[0,1]$ (see \cite{nelsen:2007}).
	\end{proof}
In the next theorem we illustrate other sufficient conditions leading to an inequality as given in Eq.\ (\ref{eq:min,prop2, cor}).
\begin{theorem}\label{prop:bound}
	Let $(X,X_\alpha)$ be the random vector with survival copula $\widehat{C}_\theta\in\mathcal{\hat{C}}^\theta$.
	Suppose that $h_{\alpha}(u)$ satisfies Assumptions 2.1 and 2.2 for all $u\in[0,1]$. 
	If Assumption 2.4 holds and if
	\begin{equation}\label{eq:prop bound}
		\int_{\frac12}^{1}\tilde{q}(u)(4u-3){\rm d}u\ge 0,
	\end{equation}
	then 
	\begin{equation*}
		\nu_X(\theta,\alpha_I)\le\,\mathsf{E}(X)-l,\qquad\forall\theta\in\Theta.		
	\end{equation*}
\end{theorem}
\begin{proof}
	From Assumptions 2.1 and 2.2, making use of Eq.\ (\ref{variability measure, d.q.d.f.}), one has 
	$$
	\begin{aligned}
		\mathsf{E}(X)-l-\nu_X(\theta,\alpha_I)&=\int_{0}^{1}\tilde{q}(u)\left[2\widehat{C}_\theta(u,u)-u\right]{\rm d}u\\
		&\ge\int_{\frac12}^{1}\tilde{q}(u)(3u-2){\rm d}u-\int_{\frac12}^{1}\tilde{q}(1-u)(1-u){\rm d}u\\
		&\ge\int_{\frac12}^{1}\tilde{q}(u)(4u-3){\rm d}u,
	\end{aligned}
	$$
	where the inequalities come from Eq.\ (\ref{eq:bound diagonal}) and Assumption 2.4, respectively. 
	Finally, the thesis follows from Eq.\ (\ref{eq:prop bound}).	
\end{proof}
\begin{example}
	Let $X$ be a random variable with s.f.\ $\overline{F}(x)=1-x^2$, for $x\in[0,1]$ and d.q.d.f.\ $\tilde{q}(u)=1/{2(1-u)^{\frac12}}$, for $u\in[0,1]$.
	Let $X_\alpha\in\mathcal{F}_{X,h_\alpha}$, with $h_\alpha\in\mathcal{H}_\alpha$ satisfying the power hazard model {\rm (iv)} of Table \ref{distortion functions}.
	Since $\overline{F}^{-1}(u)=(1-u)^{\frac12}$, for $u\in[0,1]$, one has 
	$$
	h_\alpha(u)=1-(1-u)^{\alpha},\qquad u\in[0,1].
	$$
	Consider the Farlie-Gumbel-Morgenstern (FGM) family of survival copulas, that is expressed by
	$$
	\widehat{C}_\theta(u,v)=uv\left(1+\theta(1-u)(1-v)\right),\qquad\theta\in[-1,1].
	$$
	Therefore, Eq.\ (\ref{variability measure, d.q.d.f.}) becomes
	\begin{equation}\label{variability measure, es bound}
		\nu_X(\theta,\alpha)= \frac13 + \frac{1}{1+2\alpha}-\frac{2}{3+2\alpha}-\frac{16\alpha(4 + 3\alpha)\theta}{(3 + 2\alpha)(5 + 2\alpha)(3 + 4\alpha)(5 + 4\alpha)},\qquad\theta\in[-1,1],\,\alpha\in\mathcal{I}.
	\end{equation}
	\noindent
	Since $h_\alpha$ satisfies Eq.\ (\ref{eq:min,prop3}) and since Eq.\ (\ref{eq:prop bound}) holds, there exists the minimum for the $\nu_X(\theta,\alpha)$ given in Eq.\ (\ref{variability measure, es bound}). This is confirmed by the plot in Figure \ref{fig:variability measure, es bound}.
\end{example}
\begin{figure}[t] 
	\begin{center}
		\includegraphics[scale=0.5]{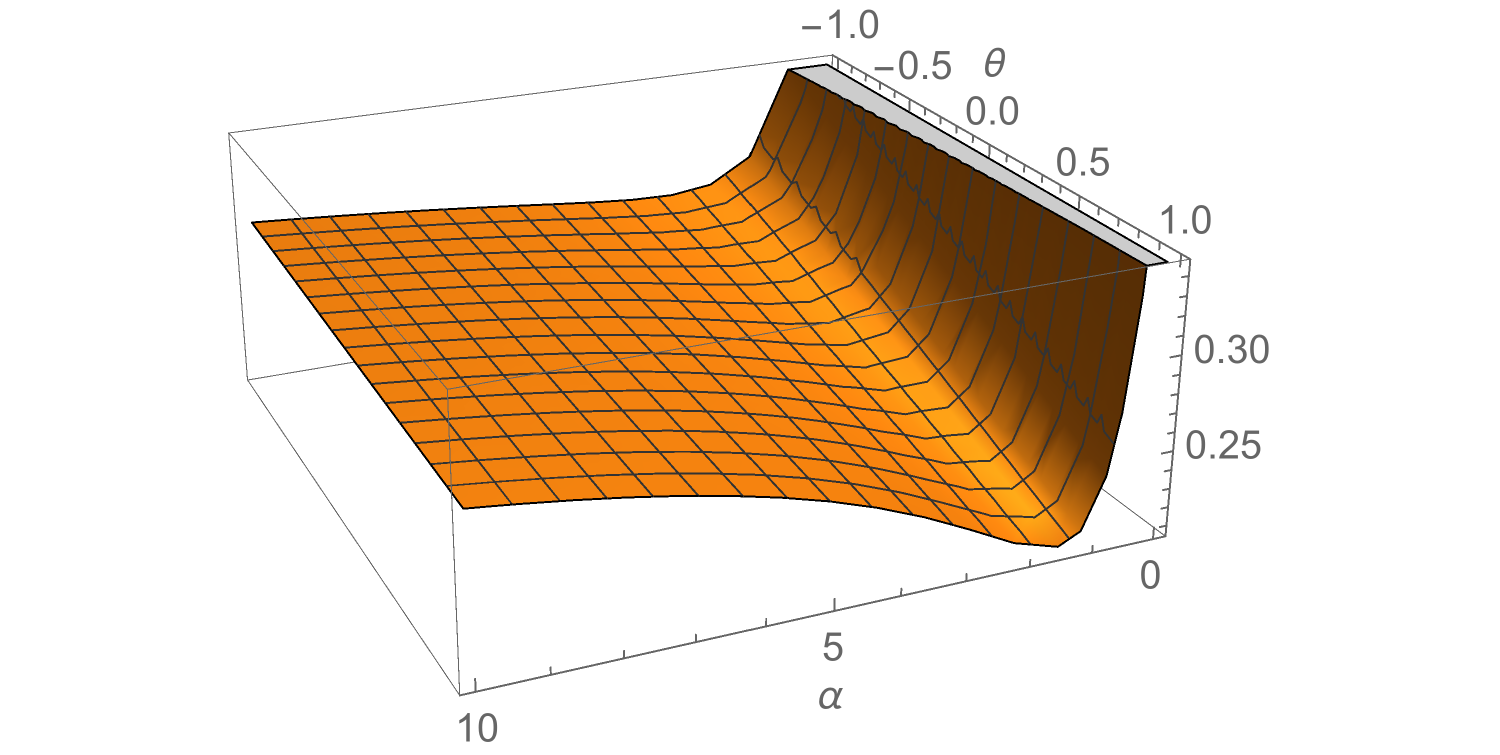}
	\end{center}
	\caption{\small Plot of $\nu_X(\theta,\alpha)$ given in Eq.\ (\ref{variability measure, es bound}) for $\theta\in[-1,1]$ and $\alpha\in[0,10]$.}
	\label{fig:variability measure, es bound}
\end{figure}

\section{Conclusions and future purposes}\label{sec: conclusions}
	In this paper we generalized the Gini's mean difference in both the directions of the distributions and of the dependences.
	We have studied sufficient conditions in order to determine the minimum, or the maximum, for the corresponding new measures of distance.
	We have also illustrated some examples in the context of hazard models.
	\par  
	Future developments can be oriented to study the measure $\nu_X(\theta,\alpha)$ for fixed distortions, in order to focus on the dependence expressed by the copula, with special attention to the case when $\alpha =\alpha_I$, i.e., when $X$ and $X_{\alpha}$  
	are identically distributed. 
	\par  
	It is worth noting that the Gini's index $\mathsf{G}(X)$, given in Eq.\ (\ref{eq:rmd,gini index}), can be suitably extended to the pair of (possibly dependent) random variables $X$ and $X_\alpha$ studied so far. In this respect, under the same assumptions of Eq.\ (\ref{eq:copula-distorted GMD}), we define the {\em copula-distorted Gini's index of $X$} as 
	\begin{equation}\label{eq:copula-distorted Gini index}
		\mathcal{GI}_X(\theta,\alpha)=\frac{\nu_X(\theta,\alpha)}{\mathsf{E}(X)+\mathsf{E}(X_\alpha)},
		\qquad\theta\in\Theta,\; \alpha\in\mathcal{I}.
	\end{equation}
	Clearly, if $X$ and $X_\alpha$ are independent, due to Eqs.\ (\ref{eq:nu=eta}), we refer to the ``distorted Gini's index of $X$",  which is obtained by replacing $\nu_X(\theta, \alpha)$ with $\eta_X(\alpha)$ in the right-hand-side of (\ref{eq:copula-distorted Gini index}). 
	The analysis and applications of the relative measure introduced in (\ref{eq:copula-distorted Gini index}) may be the object of forthcoming investigations.

\section*{Declaration of Competing Interest} 
The authors declare that they have no known competing financial interests or personal relationships that could have appeared to influence the work reported in this paper.

\section*{Data availability} 
No data was used for the research described in the article.

\section*{Acknowledgements} 
{\small
	A.D.C. and M.C. are members of the group GNCS of INdAM (Istituto Nazionale di Alta Matematica).
	F.P. is member of the group GNAMPA of INdAM.
	This work is partially supported by PRIN 2022, project “Anomalous Phenomena on Regular and Irregular Domains: Approximating Complexity for the Applied Sciences", and PRIN 2022 PNRR, project “Stochastic Models in Biomathematics and Applications". F.P. carried out this study within the FAIR - Future Artificial Intelligence Research and received funding from the European Union Next-GenerationEU (Piano Nazionale di Ripresa e Resilienza (PNRR) – Missione 4 Componente 2, Investimento 1.3 – D.D. 1555 11/10/2022, PE00000013). 
}


\begin{thebibliography}{10}
	
	{\small
		\bibitem{Bassan:etal} Bassan, B., Denuit, M., Scarsini, M., 1999. Variability orders and mean differences. Stat. Prob. Lett. {\bf 45}, 121-130.
		\bibitem{Bebbington:etal} Bebbington, M., Lai, C.D., Zitikis, R., 2008. Reduction in mean residual life in the presence of a constant competing risk. Appl. Stoch. Models Bus. Ind. 
		{\bf 24}(1), 51-63.
		%
		\bibitem{bernard:muller} Bernard, C., M\"uller, A., 2020. Dependence uncertainty bounds for the energy score and the multivariate Gini mean difference. Depend. Modeling {\bf 8}, 239-253.
		%
		\bibitem{Bickel:Lehmann} Bickel, P.J., Lehmann, E.L., 2012. Descriptive statistics for nonparametric models IV. Spread. Selected Works of E. L. Lehmann (Rojo J., ed.). Selected Works in Probability and Statistics. Boston: Springer, 519-526.
		%
		\bibitem{Capaldo:etal} Capaldo, M., Di Crescenzo, A., Meoli, A., 2022.
		Cumulative information generating function and generalized Gini functions, http://arxiv.org/abs/2307.14290. 
		%
		\bibitem{DelbaenMajumdar} Delbaen, F., Majumdar, C., 2023.
		Approximation with independent variables. Front. Math. Finance {\bf 2}(2), 141-149.  
		%
		\bibitem{dicre:2000} Di Crescenzo, A., 2000. Some results on the proportional reversed hazards model. 
		Stat. Prob. Lett. {\bf 50}, 313-321.
		%
		\bibitem{Gupta:2007} Gupta, R.C., Gupta, R.D., 2007. Proportional reversed hazard rate model and its applications. Jour. Stat. Plann. Inf. {\bf 137}(11), 3525-3536.
		%
		\bibitem{Kattumannil:etal} Kattumannil, S.K., Sreedevi, E.P., Balakrishnan, N., 2023. Relationships between cumulative entropy/extropy, Gini
		mean difference and probability weighted moments. Prob. Eng. Inf. Sci., 1–11.
		%
		\bibitem{Kumar:1994} Kumar, D., Klefsj\"o, B., 1994. Proportional hazards model: a review. Reliab. Eng. Syst. Saf. {\bf 44}, 177-188.
		%
		\bibitem{navarro:2022} Navarro, J., 2022. Introduction to System Reliability Theory. Springer.
		%
		\bibitem{nelsen:2007} Nelsen, R.B., 2007. An Introduction to Copulas. Springer Science \& Business Media.
		%
		\bibitem{patricia:2023} Ortega-Jim\'enez, P., Pellerey, F., Sordo, M.A., Su\'arez-Llorens, A., 2023. A minimizing problem of distances between random variables with proportional reversed hazard rate functions. 
		In Building Bridges between Soft and Statistical Methodologies for Data Science. Springer, Cham., 311-318.
		%
		\bibitem{patricia:2021} Ortega-Jim\'enez, P., Sordo, M.A., Su\'arez-Llorens, A., 2021. Stochastic comparisons of some distances between random variables. Mathematics {\bf 9}(9), 981.
		%
		\bibitem{shaked:2007} Shaked, M., Shanthikumar, J.G., 2007. Stochastic Orders and Their Applications. San Diego: Academic Press.
		%
		\bibitem{Sunoj:2012}  Sunoj, S.M., Sankaran, P.G., 2012. Quantile based entropy function, Stat. Prob. Lett. {\bf 82}, 1049-1053.
		%
		\bibitem{vila:2023}  Vila, R., Balakrishnan, N., Saulo, H., 2023. An upper bound and a characterization for Gini’s mean difference based on correlated random variables,
		https://doi.org/10.48550/arXiv.2301.07229.
	}
	
\end{thebibliography}
\end{document}